\documentclass[11pt]{amsart}
\usepackage{mathrsfs}
\usepackage{CJK,fancyhdr,amscd}
\usepackage{anysize}
\usepackage{amsmath,amssymb,amsfonts}
\usepackage{epsfig}
\usepackage{amsthm}
\usepackage{latexsym}
\usepackage{indentfirst, latexsym, bm}
\usepackage{graphics}
\usepackage{bbding}
\usepackage{dsfont}
\usepackage[pagebackref]{hyperref}
\usepackage[all,poly,knot]{xy}
\numberwithin{equation}{section}

\newcommand{\comment}[1]{}

\usepackage{fancyhdr}
\newtheorem{theo} {Theorem} [section]
\newtheorem{prop}[theo]{Proposition}
\newtheorem{coro}  [theo]     {Corollary}
\newtheorem{lemm}  [theo]     {Lemma}

\newtheorem{defi}  [theo]     {Definition}

\newcommand{\End}{\textmd{End}}
\def\mc{\mathcal}
\newcommand{\db}{\overline{\partial}}
\newcommand{\G}{\mathbb{G}}
\newcommand{\p}{\partial}
\def\b{\bar}

\begin{document}

\title[Deformation of CR-structures]{On Tian-Todorov lemma and its applications to deformation of CR-structures}

\author{Sheng Rao}
\address{Department of Mathematics and Statistics, Wuhan University, Wuhan 430072, China}
\email{likeanyone@whu.edu.cn}


\author{Yongpan Zou}
\address{Department of Mathematics and Statistics, Wuhan University, Wuhan 430072, China}
\email{yongpan\_zou@whu.edu.cn}



\thanks{Rao is partially supported by NSFC (Grant No. 11671305, 11771339).}

\begin{abstract}
We give a new Tian-Todorov lemma on deformations of CR-structures and use it to reprove the deformation unobstructedness of normal compact strongly pseudoconvex CR-manifold under the assumption of $d'd''$-lemma, more faithfully following Tian-Todorov's approach.

\end{abstract}

\maketitle


%
\section{Introduction}
  It is a celebrated result in deformation theory that the deformations of Calabi-Yau manifold are unobstructed, by the Bogomolov-Tian-Todorov theorem, due to F. Bogomolov \cite{[B]}, G. Tian \cite{[T]} and A. Todorov \cite{[T1]}. Z. Ran and Y. Kawamata give pure algebraic proofs in \cite{kaw,ran} independently and see also another proof in \cite{C}. The remarkable generalization of Bogomolov-Tian-Todorov theorem to the logarithmic case is due to L. Katzarkov-M. Kontsevich-T. Pantev \cite{KKP08} and D. Iacono \cite{Iacono} by Deligne's spectral sequence (see also a more analytic argument in \cite{lrw}). In \cite{[P]}, D. Popovici proved the Kuranishi family of Calabi-Yau $\partial\overline{\partial}$-manifolds is unobstructed. In Popovici's paper, a compact complex manifold $X$ will be said to be a \emph{$\partial\overline{\partial}$-manifold} if the \emph{$\partial\overline{\partial}$-lemma} holds on $X$, which means for every pure-type $d$-closed form on a complex manifold, the properties of
$d$-exactness, $\p$-exactness, $\b{\p}$-exactness and
$\p\b{\p}$-exactness are equivalent on $X$. Obviously, the $\partial\overline{\partial}$-lemma is a weaker property than the K\"ahlerness.

Moreover, for the deformation of CR-structures, T. Akahori-K. Miyajima obtain a corresponding result in \cite{[AM]}. Firstly, they proved the so called ``a CR-analogue of Tian-Todorov's lemma".
\begin{lemm} [{\cite[Lemma 4.6]{[AM]}}]\label{tt}
If $ \alpha, \beta \in A^{n-2,1}(M) $ and satisfy $ d'\alpha = d'\beta = 0 $, then
$$ 2~\imath_2 [\imath_1^{-1}\alpha, \imath_1^{-1}\beta] = - d'(\imath_1^{-1}\alpha \lrcorner \beta + \imath_1^{-1}\beta \lrcorner \alpha),$$
where the bundle isomorphism $\imath_\cdot$ is given in \eqref{imath-q} and $\imath_\cdot^{-1}$ is its inverse.
\end{lemm}

Then, by using Lemma \ref{tt} and Hodge theory, Akahori-Miyajima follow the approach of Tian \cite{[T]} and Todorov \cite{[T1]} to obtain the unobstructedness for deformations of CR structures.

\begin{theo} [{\cite[MAIN THEOREM]{[AM]}}]
Let $ (M,\ ^0T'') $ be a normal strongly pseudoconvex manifold with $ \dim_{\mathbb{R}}M  = 2n-1 \geq 7 $. And we assume that its canonical line bundle $ K_M = \wedge^n(T')^* $ is trivial in CR-sense. Then the obstructions in $ \imath_1^{-1}(\mathbf{Z}^1) $ appear in $ \mathbf{J}^{n-2,2} $. That is, if $ \mathbf{J}^{n-2,2}=0 $, then any deformation of CR structures in $ \imath_1^{-1}(\mathbf{Z}^1) $ is unobstructed. Here $\mathbf{Z}^1$ and $ \mathbf{J}^{n-2,2} $ are given by \eqref{bf-z} and \eqref{assume}, respectively.
\end{theo}

In our paper, the first goal is to present a new approach to generalize Lemma \ref{tt} by using a twisted commutator formula on generalized complex manifolds. We obtain:
\begin{lemm} [= Lemma \ref{TI-TO}]
If $ \alpha, \beta \in A^{n-2,1}(M) $ and $ \rho \in \Gamma(M, \wedge ^{\bullet}(T')^*) $, we have
$$ [\imath_1^{-1}\alpha, \imath_1^{-1}\beta]\lrcorner\rho = \imath_1^{-1}\alpha\lrcorner d'(\imath_1^{-1}\beta\lrcorner\rho) + \imath_1^{-1}\beta\lrcorner d'(\imath_1^{-1}\alpha\lrcorner\rho) - d'(\imath_1^{-1}\beta\lrcorner \imath_1^{-1}\alpha\lrcorner\rho) - \imath_1^{-1}\beta\lrcorner(\imath_1^{-1}\alpha\lrcorner d'\rho). $$
\end{lemm}

Then we reprove the main theorem by the iteration methods more faithfully along the lines of the approach in Tian \cite{[T]} and Todorov \cite{[T1]} after we establish the following crucial lemma.
\begin{lemm} [= Lemma \ref{5.4}]
Assume that $ \phi_k \in \mathbf{Z}^1$ satisfy
$$ \overline{\partial}_b(\imath_1^{-1}\phi_k) = -\frac{1}{2} \sum\limits_{m+h=k} [\imath_1^{-1}\phi_m, \imath_1^{-1}\phi_h],$$
for any $k=2,\cdot\cdot\cdot,l.$  Moreover,
$$ \overline{\partial}_b(\imath_1^{-1}\phi_1)=0.$$
Then one has
$$ \overline{\partial}_b\left(\sum\limits_{m+h=l+1} [\imath_1^{-1}\phi_m, \imath_1^{-1}\phi_h]\right) = 0. $$
\end{lemm}

As Miyajima proposed in \cite{miy}, it is an interesting to present a non-trivial example of normal strongly pseudo-convex CR manifold on which the $d'd''$-lemma holds.
Moreover, we believe that it is worth studying the analogous criteria of the $\partial\bar{\partial}$-lemma for a normal strongly pseudo-convex CR manifold to the works \cite{[AT13],atr} under some suitable assumptions on the manifolds.


In this paper, we will follow the notations and definitions in \cite{[AM]} and \cite{[LR]}. Without specially mentioned, all the subindexes are assumed to be nonnegative integers.

\section{Preliminaries}
\subsection{CR structures}
  Let $ M $ be a $ C^\infty $ orientable real odd dimensional manifold. Let $ E $ ba a subbundle of the complexified tangent bundle $ \mathbb{C}\otimes TM $ satisfying:
$$
E \cap \overline{E}=0,
$$
then $E$ is called an \emph{almost CR structure} on $M$. An almost CR structure $E$ on $M$ is \emph{integrable} if for any open set $U \subseteq M$,
$$
[\Gamma(U,~E), \Gamma(U,~E)] \subset \Gamma(U,~E).
$$

An integrable almost CR structure is referred to as a \emph{CR structure}, and a pair $(M,~E)$ consisting of a $ C^{\infty}$ manifold and a CR structure is a \emph{CR manifold}. It is well known in CR geometry that if $M$ is a real hypersurface of $X$, then $^{0}T'':=\mathbb{C}\otimes T(M)\cap T^{0,1}(X)|_{M}$ is a CR structure.

  Let $X$ be a complex manifold and let $r$ be a $ C^\infty$ exhaustion function on $X$ which is strictly plurisubharmonic except a compact subset of $X$.
Let $$ \Omega=\{x: x\in X, r(x)<0\} $$
and assume that the boundary $b\Omega$ of $\Omega$ is smooth. Then we can naturally put a CR-structure over $b\Omega$. Namely, we set
$$  ^{0}T''=\mathbb{C}\otimes T(b\Omega)\cap T^{0,1}(X)|_{b\Omega}. $$
Then we have
$$
^0T'' \cap~^0\overline{T}''=0,
$$
and
$$
[\Gamma(b\Omega,~^0T''), \Gamma(b\Omega,~^0T'')] \subset \Gamma(b\Omega,~^0T'').
$$
For our pair $ (b\Omega,\ ^0T'') $, we set a $ C^\infty $ vector bundle isomorphism
\begin{equation}\label{dec}
\mathbb{C}\otimes T(b\Omega)=\ ^0T'' +\ ^0\overline{T}'' + \mathbb{C}\zeta,
 \end{equation}
 where $ \zeta $ is a real vector field supplement to $~^0T'' +\ ^0\overline{T}'' $.
 This decomposition gives rise to a Levi form $L$ as: for $X,Y\in\ ^0T''$,
 $$L(X,Y)=-\sqrt{-1}[X,\bar Y]_{\mathbb{C}\zeta},$$
 where $[X,\bar Y]_{\mathbb{C}\zeta}$ refers to the $\mathbb{C}\zeta$-part with respect to the decomposition \eqref{dec}. A CR structure is \emph{strongly pseudoconvex}
if its Levi form is (positive or negative) definite at each point.

 Let $(M,\ ^{0}T'') $ be a compact strongly pseudoconvex CR-manifold. Furthermore we assume that $ (M,\ ^0T'') $ admits a normal vector field, namely there is a global vector field $ \zeta $ on $M$ satisfying:
\begin{align} \label{2.5}
  \zeta_p \not\in\ ^0T_p'' +~^0\overline{T}_p''
\end{align}
for every point $p$ of $M$, and
\begin{align}\label{2.6}
 [\zeta, \Gamma(M,~^0T'')] \subseteqq \Gamma(M,~^0T'').
\end{align}
This manifold $ (M,\ ^{0}T'',\zeta) $ is called a \emph{normal s.p.c. manifold}.

\begin{prop}[{\cite[Proposition 1.6.1]{[A2]}}]
An almost CR-structure $ ^\phi T'' $ corresponds to an element $ \phi $ of $ \Gamma(M, T'\otimes(^0T'')^*) $ bijectively. The correspondence is that: for $\phi\in \Gamma(M, T'\otimes(^0T'')^*)$,
$$ ^\phi T^{''}=\{X' : X'=X+\phi(X),X \in\ ^0T''\},$$
where $ T' =\ ^0\overline{T}'' + \mathbb{C}\zeta $.
\end{prop}
 And we have
\begin{prop}[{\cite[Proposition 1.6.2]{[A2]}}]
An almost CR-structure $ ^\phi T'' $ is an actual CR-structure if and only if $ \phi $ satisfies the non-linear partial differential equation $ \mathcal{P}(\phi)=0.$
\end{prop}

Here $ \mathcal{P}(\psi) $ is defined as follows. For $ \psi $ in $ \Gamma(M, T'\otimes(^0T'')^*) $,
$$ \mathcal{P}(\psi)= \overline{\partial}_b\psi + \frac{1}{2}[\psi,\psi],$$
where $ \overline{\partial}_b $ means the $ T' $-valued tangential Cauchy-Riemann operator on $ M $. Here we define the \emph{tangential Cauchy-Riemann operators}
$$ \overline{\partial}_b^{(q)} : \Gamma(M,~\wedge ^q(^0T'')^*) \rightarrow \Gamma(M,~\wedge ^{q+1}(^0T'')^*) $$
by
\begin{align}
 \notag \overline{\partial}_b^{(q)}\phi(X_1, X_2,\cdot\cdot\cdot,X_{q+1})=&\sum\limits_{i=1}^{q+1}(-1)^{i-1} X_i(\phi( X_1,\cdot\cdot\cdot,\widehat{X_i}, \cdot\cdot\cdot,X_{q+1})) \\
 \notag &+ \sum\limits_{i<j}(-1)^{i+j}\phi([X_i, X_j], X_1,\cdot\cdot\cdot, \widehat{X_i}, \cdot\cdot\cdot, \widehat{X_j}, \cdot\cdot\cdot,X_{q+1}),
\end{align}
for all $ \phi \in \Gamma(M,~\wedge ^q(^0T'')^*) $ and $ X_1,\cdot\cdot\cdot,X_{q+1} \in \Gamma(M,~^0T'') $, and the Lie bracket on $ \Gamma(M,\ ^0\overline{T}''\otimes(^0T'')^*) $
as:  for $ X, Y\in \Gamma(M,\ ^0T'') $
\begin{align*}
&\ [\phi, \psi](X, Y)\\
 :=&\ [\phi(X), \psi(Y)] + [\psi(X), \phi(Y)] - \phi([X, \psi(Y)]_{^0T''} + [\psi(X), Y]_{^0T''}) - \psi([\phi(X),Y]_{^0T''} + [X, \phi(Y)]_{^0T''}).
\end{align*}

\subsection{Kuranishi family and Beltrami differentials}\label{reduction}
Here we give a rough introduction to Kuranishi family to describe the notation of deformation unobstructedness in complex geometry.

By (the proof of) Kuranishi's completeness theorem \cite{ku}, for any compact complex manifold $X_0$, there exists a complete holomorphic family
$\varpi:\mc{K}\to T$ of complex manifolds at the reference point $0\in T$ in the sense that for any differentiable family $\pi:\mc{X}\to B$ with $\pi^{-1}(s_0)=\varpi^{-1}(0)=X_0$, there is a sufficiently small neighborhood $E\subseteq B$ of $s_0$, and smooth maps $\Phi: \mathcal {X}_E\rightarrow \mathcal {K}$,  $\tau: E\rightarrow T$ with $\tau(s_0)=0$ such that the diagram commutes
$$\xymatrix{\mathcal {X}_E \ar[r]^{\Phi}\ar[d]_\pi& \mathcal {K}\ar[d]^\varpi\\
(E,s_0)\ar[r]^{\tau}  & (T,0),}$$
$\Phi$ maps $\pi^{-1}(s)$ biholomorphically onto $\varpi^{-1}(\tau(s))$ for each $s\in E$, and $$\Phi: \pi^{-1}(s_0)=X_0\rightarrow \varpi^{-1}(0)=X_0$$ is the identity map.
This family is called \emph{Kuranishi family} and constructed as follows. Let $\{\eta_\nu\}_{\nu=1}^m$ be a basis for $\mathbb{H}^{0,1}(X_0,T^{1,0}_{X_0})$, where some suitable hermitian metric is fixed on $X_0$ and $m\geq 1$; Otherwise the complex manifold $X_0$ would be \emph{rigid}, i.e., for any differentiable family $\kappa:\mc{M}\to P$ with $s_0\in P$ and $\kappa^{-1}(s_0)=X_0$, there is a neighborhood $V \subseteq P$ of $s_0$ such that $\kappa:\kappa^{-1}(V)\to V$ is trivial. Then one can construct a holomorphic family
$$\label{phi-ps-pp}\varphi(t) = \sum_{|I|=1}^{\infty}\varphi_{I}t^I:=\sum_{j=1}^{\infty}\varphi_j(t),\ I=(\imath_1,\cdots,\imath_m),\ t=(t_1,\cdots,t_m)\in \mathbb{C}^m,$$
 {for $t\in D_\rho$ a small $\rho$-disk,} of Beltrami differentials as follows:
$$\label{phi-ps-0}
 \varphi_1(t)=\sum_{\nu=1}^{m}t_\nu\eta_\nu
$$
and for $|I|\geq 2$,
$$\label{phi-ps}
  \varphi_I=\frac{1}{2}\db^*\G\sum_{J+L=I}[\varphi_J,\varphi_{L}],
$$
where $\G$ is the associated Green's operator. A \emph{Beltrami differential} of $X$ is a holomorphic
tangent bundle-valued $(0,1)$-form on $X$.
It is obvious that $\varphi(t)$ satisfies the equation
$$\varphi(t)=\varphi_1+\frac{1}{2}\db^*\G[\varphi(t),\varphi(t)].$$
Let
$$T=\{t\in D_\rho\ |\ \mathbb{H}[\varphi(t),\varphi(t)]=0 \},$$
where $\mathbb{H}$ is the associated harmonic projection.
Thus, for each $t\in T$, $\varphi(t)$ satisfies
$$\label{int}
\b{\p}\varphi(t)=\frac{1}{2}[\varphi(t),\varphi(t)],
$$
and determines a complex structure $X_t$ on the underlying differentiable manifold of $X_0$. More importantly, $\varphi(t)$ represents the complete holomorphic family $\varpi:\mc{K}\to T$ of complex manifolds. Roughly speaking, Kuranishi family $\varpi:\mc{K}\to T$ contains all sufficiently small differentiable deformations of $X_0$.
We call the analytic subset $T$ the \emph{Kuranishi space} of this Kuranishi family. Moreover, if $T=D_\rho$, the deformation of $X$ is \emph{unobstructed} (in $\varphi_1(t)$).
The deformation unobstructedness of CR-structured can be defined analogously.

\section{Twisted commutator formula on generalized complex manifolds}
In this section, we introduce a twisted commutator formula on generalized complex manifolds to prove the ``a CR-analogy of Tian-Todorov lemma" in the next section \ref{s-tt}.

First of all, let us introduce some notations on generalized complex
geometry.
Let $\check{M}$ be a smooth manifold, $T:=T_{\check{M}}$ the tangent
bundle of $\check{M}$ and $T^*:=T^*_{\check{M}}$ its cotangent
bundle. In the generalized complex geometry, for any $X,Y\in
C^\infty(T)$ and $\xi,\eta\in C^\infty(T^*)$, $T\oplus T^*$ is
endowed with a \emph{canonical} nondegenerate \emph{inner product}
given by
$$\label{cinner}
\langle
X+\xi,Y+\eta\rangle=\frac{1}{2}\big(\iota_X(\eta)+\iota_Y(\xi)\big),
$$
 where we denote by $\iota_X$ the contraction of a differential form by the vector field $X$. And there is an important canonical bracket on $T\oplus T^*$, so-called
\emph{Courant bracket}, which is defined by
\begin{equation}\label{Courant bracket}
[X+\xi,Y+\eta]=[X,Y]+L_X\eta-L_Y\xi-\frac{1}{2}d\big(\iota_X(\eta)-\iota_Y(\xi)\big).
\end{equation}
Here, we denote by $L_X$ the Lie derivative and $[\cdot,\cdot]$ on the right-hand side is the ordinary Lie
bracket of vector fields. Note that on vector fields the Courant
bracket reduces to the Lie bracket; in other words, if $pr_1:T\oplus
T^*\rightarrow T$ is the natural projection,
$$pr_1([A,B])=[pr_1(A),pr_1(B)]),$$
for any $A,B\in C^\infty(T\oplus T^*).$

A \emph{generalized almost complex structure} on $\check{M}$ is a
smooth section $J$ of the endomorphism bundle $\End(T\oplus T^*)$,
which satisfies both symplectic and complex conditions, i.e.
$J^*=-J$ and $J^2=-1$. We can show that the
obstruction to the existence of a generalized almost complex
structure is the same as that for an almost complex structure (See
\cite[Proposition 4.15]{[G]}). Hence it is obvious that
(generalized) almost complex structures only exist on the
even-dimensional manifolds. Let $E\subset(T\oplus T^*)\otimes
\mathbb{C}$ be the $+i$-eigenbundle of the generalized almost
complex structure $J$. Then if $E$ is Courant involutive, i.e.
closed under the Courant bracket (\ref{Courant bracket}), we say
that $J$ is \emph{integrable} and also a \emph{generalized complex
structure}. Note that $E$ is a maximal isotropic subbundle of
$(T\oplus T^*)\otimes \mathbb{C}$.

As observed by P. \v{S}evera-A. Weinstein \cite{[SW]}, the
Courant bracket (\ref{Courant bracket}) on $T\oplus T^*$ can be
twisted by a real, closed $3$-form $H$ on $\check{M}$ in the
following way: given $H$ as above, define another important bracket
$[\cdot,\cdot]_H$ on $T\oplus T^*$ by
$$[X+\xi,Y+\eta]_H=[X+\xi,Y+\eta]+\iota_Y\iota_X (H),$$
which is called an \emph{$H$-twisted Courant bracket}.
\begin{defi} \label{}
A generalized complex structure $J$ is said to be \emph{twisted
generalized complex} with respect to the closed $3$-form $H$ when
its $+i$-eigenbundle $E$ is involutive with respect to the
$H$-twisted Courant bracket and then the pair $(\check{M},J)$ is
called an \emph{$H$-twisted generalized complex manifold}.
\end{defi}

From now on, we consider the $H$-twisted generalized complex
manifold $(\check{M},J)$ defined as above. Postponing listing some
more notions in need, we must remark that they are not exactly the
same as the usual ones since we just define them for our
presentation below, and possibly miss their usual geometrical meaning.
The
 \emph{twisted de Rham differential} is given by
$$d_R=d+(-1)^k R\wedge\cdot,$$
where $R\in \Omega^k(\check{M},\mathbb{R})$. For any
$X\in C^\infty(T),\ \xi\in C^\infty(T^*)$ and $\alpha\in
\Omega^*(\check{M},\mathbb{C})$, a natural action of
$T\oplus T^*$ on smooth differential forms is given by
$$(X+\xi)\cdot\alpha=\iota_X(\alpha)+\xi\wedge\alpha.$$
 Actually, this action can be
considered as \lq lowest level' of a hierarchy of actions on the
bundles $T\bigoplus(\oplus_r\wedge^r T^*)$, $r=1,2,\cdots$, defined
by the similar formula
$$(X+\xi_1+\xi_2+\cdots)\cdot\alpha=\iota_X(\alpha)+\xi_1\wedge\alpha+\xi_2\wedge\alpha+\cdots,$$
for any $X\in C^\infty(T),\ \xi_1+\xi_2+\cdots\in
C^\infty(\oplus_r\wedge^r T^*)$ and $\alpha\in
\Omega^*(\check{M},\mathbb{C})$. Then in what follows we
adopt the action of $A=A_1\wedge\cdots\wedge A_k\in
C^\infty\Big(\bigwedge^k \big(T\bigoplus(\oplus_r\wedge^r
T^*)\big)\Big)$ on $\Omega^*(\check{M},\mathbb{C})$ given by
\begin{equation}\label{action}
A\cdot\alpha=(A_1\wedge\cdots\wedge A_k)\cdot\alpha\equiv A_1\cdot
A_2\cdot\cdots \cdot A_k\cdot\alpha,\qquad\text{for any $\alpha\in
\Omega^*(\check{M},\mathbb{C})$}.
\end{equation}
The \emph{generalized Schouten bracket} for $A=A_1\wedge\cdots\wedge
A_p\in C^\infty(\wedge^p (T\oplus T^*))$ and
$B=B_1\wedge\cdots\wedge B_q\in C^\infty(\wedge^q (T\oplus T^*))$ is
defined as
$$[A,B]_R=\sum_{i,j}(-1)^{i+j}[A_i,B_j]_R\wedge A_1\wedge\cdots\wedge \hat{A}_i\wedge
\cdots\wedge A_p\wedge
B_1\wedge\cdots\wedge\hat{B}_j\wedge\cdots\wedge B_q,$$
where $\hat{}$ means \lq omission', the \emph{$R$-twisted Courant bracket}
 $$[A_i,B_j]_R$$ is defined as $$[A_i,B_j]+\iota_{Y_j}\iota_{X_i}(R)$$ if
we take $A_i=X_i+\xi_i$ and $B_j=Y_j+\eta_j$, and the action of
$[A_i,B_j]_R$ comply with the principle of (\ref{action}). Here we
note that if $R$ is a $3$-form and $X+\xi$, $Y+\eta\in
C^\infty(T\oplus T^*)$, then the $R$-twisted Courant bracket
$[X+\xi,Y+\eta]_R$ still lies in $C^\infty(T\oplus T^*)$. However,
for $R$ being general, the bracket $[X+\xi,Y+\eta]_R$ doesn't lie in
$C^\infty(T\oplus T^*)$ in general since $\iota_Y\iota_X (R)$ is not
necessarily a $1$-form, but in $C^\infty(T\bigoplus(\oplus\wedge^*
T^*))$; hence this bracket still makes sense under the action
(\ref{action}).
Then we have the useful twisted commutator formula.
\begin{prop}\label{3.2}
  For any smooth
differential form $\rho$, any smooth odd-degree form $R$ and any
$A\in C^\infty(\wedge^p E^*)$, $B\in C^\infty(\wedge^q E^*)$, we
have
$$\label{Ht}
d_R(A\cdot B\cdot\rho)=(-1)^pA\cdot
d_R(B\cdot\rho)+(-1)^{(p-1)q}B\cdot
d_R(A\cdot\rho)+(-1)^{p-1}[A,B]_R\cdot\rho+(-1)^{p+q+1}A\cdot B\cdot
d_R\rho.
$$
\end{prop}
\begin{proof}
See \cite[Proposition $4.2$]{[LR]}, \cite[Lemma $4.24$]{[G]},
\cite[($17$)]{[KL]} and \cite[Lemma $2$]{[L]}.
\end{proof}

As a direct corollary of Proposition \ref{3.2}, one has

\begin{coro}\label{H1AB2}
For any smooth differential form $\rho$, any smooth $1$-form $R$ and
any $A,B\in C^\infty(\wedge^2 E^*)$, we have
\begin{equation}\label{iH1AB2}
d_R(A\cdot B\cdot\rho)=A\cdot d_R(B\cdot\rho)+B\cdot
d_R(A\cdot\rho)-[A,B]\cdot\rho-A\cdot B\cdot d_R\rho.
\end{equation}
\end{coro}

\section{CR analogue of Tian-Todorov's lemma}\label{s-tt}
We now focus on the basic operators in CR geometry, whose definitions and notations conform to \cite{[AM]}.

As proved in \cite{[A1]}, if we let $(M, E)$ be an abstract strongly pseudoconvex CR-structure with $ \dim_{\mathbb{R}}M  = 2n-1 \geq 7 $, the vector bundle $ T' =\ ^0\overline{T}'' + \mathbb{C}\zeta $ will be a CR-holomorphic vector bundle. Therefore we can introduce the canonical line bundle $\wedge ^n(T')^*$ like in the complex manifold case.

Now we introduce $ d'' $ operator on $\Gamma (M, (\mathbb{C}\zeta)^*\wedge\wedge ^p(^0\overline{T}'')^*\wedge\wedge ^q(^0T'')^*)$.
Namely, for $ u $ in $ \Gamma (M, (\mathbb{C}\zeta)^*\wedge\wedge ^p(^0\overline{T}'')^*\wedge\wedge ^q(^0T'')^*),$
$$ d''u = (du)_{(\mathbb{C}\zeta)^*\wedge\wedge  ^p(^0\overline{T}'')^*\wedge\wedge ^{q+1}(^0T'')^*}, $$
where $ (du)_{(\mathbb{C}\zeta)^*\wedge\wedge  ^p(^0\overline{T}'')^*\wedge\wedge ^{q+1}(^0T'')^*} $ means the $ (\mathbb{C}\zeta)^*\wedge\wedge  ^p(^0\overline{T}'')^* $ part of $ du $ according to the vector bundle decomposition.

Similarly, for any $ u\in \Gamma (M, (\mathbb{C}\zeta)^*\wedge\wedge  ^p(^0\overline{T}'')^*\wedge\wedge ^q(^0T'')^*) $, we set
$$ d'u= (du)_{(\mathbb{C}\zeta)^*\wedge\wedge  ^{p+1}(^0\overline{T}'')^*\wedge\wedge ^q(^0T'')^*} $$
and thus,
$$ du= d''u + d'u + (du)_{\wedge^{p+1}(^0\overline{T}'')^*\bigwedge\wedge ^{q+1}(^0T'')^*}. $$
By a direct calculation, we have
$$ (du)_{\wedge^{p+1}(^0\overline{T}'')^*\wedge\wedge ^{q+1}(^0T'')^*} = -d\theta \wedge (\zeta \lrcorner u), $$
where $ \theta $ is the $1$-form defined by $ \theta|_{^0T''+ ^0\overline{T}''} =0 $ and $ \theta(\zeta)=1$. Here we denote by $\lrcorner$ the contraction operator. We will see the relation between these operators. For arbitrary $ u$ as above,
$$ du = d'u + d''u - d\theta\wedge(\zeta \lrcorner u). $$
And so
\begin{align*}
ddu =&\ d'd'u + d''d'u + d'd''u\\
 &\ - (d(d\theta\wedge(\zeta \lrcorner u)))_{(\mathbb{C}\zeta)^*\wedge\wedge  ^{p+1}(^0\overline{T}'')^*\wedge\wedge ^{q+1}(^0T'')^*}\\
 &\ + d''d''u - d\theta\wedge(\zeta \lrcorner d'u)- d'(d\theta\wedge(\zeta \lrcorner u))\\
 &\ - d\theta\wedge(\zeta \lrcorner d''u) - d''(d\theta\wedge(\zeta \lrcorner u)).
\end{align*}
By comparing the type, we have the following relations. Namely, from the part in $ (\mathbb{C}\zeta)^*\wedge\wedge  ^{p+2}(^0\overline{T}'')^*\wedge\wedge ^q(^0T'')^* $,
$$ d'd'u = 0. $$
From the part in $ (\mathbb{C}\zeta)^*\wedge\wedge  ^p(^0\overline{T}'')^*\wedge\wedge ^{q+2}(^0T'')^* $,
$$ d''d''u=0. $$
From the part in $ \wedge  ^{p+2}(^0\overline{T}'')^*\wedge\wedge ^{q+1}(^0T'')^* $,
\begin{align}\label{eq4.1}
 d\theta\wedge (\zeta \lrcorner d'u)+ d'(d\theta\wedge (\zeta \lrcorner u))=0.
\end{align}
From the part in $ \wedge  ^{p+1}(^0\overline{T}'')^*\wedge\wedge ^{q+2}(^0T'')^* $,
\begin{align}\label{eq4.2}
d\theta\wedge (\zeta \lrcorner d''u)+ d''(d\theta\wedge (\zeta \lrcorner u))=0.
\end{align}

Let $ (M,\ ^0T'') $ be a normal s.p.c. manifold with a real vector field $ \zeta $ satisfying \eqref{2.5} and \eqref{2.6} and with $ \dim_\mathbb{R}M=2n-1\geq7$. In this section, we will assume that the canonical line bundle $ K_M = \wedge^n(T')^* $ is \emph{trivial in CR-sense}, that is there exists a nowhere vanishing section $ \omega\in\Gamma(M,\wedge^n(T')^*)$ satisfying $ d''\omega=0 $.

In this paper, we will consider a bundle isomorphism
\begin{equation}\label{imath-q}
\imath_q: T'\otimes \wedge^q(^0T'')^* \rightarrow \wedge^{n-1}(T')^*\wedge\wedge^q(^0T'')^*
\end{equation}
 given by
 $$\imath_q(u)= u\lrcorner\omega.$$
Note that
$$ \imath_q(^0\overline{T}''\otimes\wedge^q(^0T'')^*) = (\mathbb{C}\zeta)^*\wedge\wedge^{n-2}(^0\overline{T}'')^*\wedge\wedge^q(^0T'')^* $$
holds.
  Now we have the Lie bracket on $ \Gamma(M,\ ^0\overline{T}''\otimes(^0T'')^*) $ is given by
\begin{align*}
&\ [\phi, \psi](X, Y)\\
 :=&\ [\phi(X), \psi(Y)] + [\psi(X), \phi(Y)] - \phi([X, \psi(Y)]_{^0T''} + [\psi(X), Y]_{^0T''}) - \psi([\phi(X),Y]_{^0T''} + [X, \phi(Y)]_{^0T''}),
\end{align*}
for $ X, Y\in \Gamma(M,\ ^0T'') $. And then a Lie bracket is induced on $ \Gamma(M, (\mathbb{C}\zeta)^*\wedge\wedge^{n-2}(^0\overline{T}'')^*\wedge(^0T'')^*) $ by
$$ [\alpha, \beta]:= \imath_2[\imath_1^{-1}\alpha, \imath_1^{-1}\beta]. $$

The main purpose of this paper is to obtain a CR-analogue of Tian-Todorov's lemma analyzing this induced Lie bracket.
We will use the following two lemmata.
\begin{lemm}[{\cite[Lemma 4.1]{[AM]}}]\label{4.1}
For any point $p\in M $, there exists a local frame $\{e_1, e_2,\cdots,e_{n-1}\} $ of $ ^0T'' $ around $ p $ satisfying
\begin{itemize}
  \item [(1)]
         $[e_i, e_j](p)=0 ~~(i,j=1,2,\cdot\cdot\cdot,n-1)$,
  \item [(2)]
          $[\overline{e}_i, e_k](p)=\sqrt{-1}\delta_{ik}\zeta_p ~~(i,k=1,2,\cdot\cdot\cdot,n-1)$.
\end{itemize}
\end{lemm}

\begin{lemm}[{\cite[Lemma 4.5]{[AM]}}] \label{4.2}
For $$ \phi = \sum\limits_{i=1}^{n-1}\sum\limits_{k=1}^{n-1} \phi^{i,k} \overline{e}_i\otimes e_k^* $$ and $$  \psi = \sum\limits_{i=1}^{n-1}\sum\limits_{k=1}^{n-1} \psi^{i,k} \overline{e}_i\otimes e_k^* \in \Gamma(M,\ ^0\overline{T}''\otimes (^0T'')^*), $$
one has
 \begin{align*}
[\phi,\psi](p)=\frac{1}{2}\sum\limits_{i=1}^{n-1}\sum\limits_{k,l=1}^{n-1}\sum\limits_{j=1}^{n-1}\{&\phi^{j,k}(p)\overline{e}_j(\psi^{i,l})(p)-\psi^{j,l}(p)\overline{e}_j(\phi^{i,k})(p)\\
 &+\psi^{j,k}(p)\overline{e}_j(\psi^{i,l})(p)- \phi^{j,l}(p)\overline{e}_j(\psi^{i,k})(p)\}(\overline{e}_i)_p\otimes(e_k^*)_p\wedge(e_l^*)_p.
\end{align*}

\end{lemm}

From now on, we use the notation $ A^{p,q}(M) $ to denote $$ A^{p,q}(M) = \Gamma(M, (\mathbb{C}\zeta)^*\wedge\wedge^p(^0\overline{T}'')^*\wedge\wedge^q(^0T'')^*).$$

\begin{lemm} \label{TI-TO}

If $ \alpha, \beta \in A^{n-2,1}(M) $ and $ \rho \in \Gamma(M, \wedge ^{\bullet}(T')^*) $, we have
$$ [\imath_1^{-1}\alpha, \imath_1^{-1}\beta]\lrcorner\rho = \imath_1^{-1}\alpha\lrcorner d'(\imath_1^{-1}\beta\lrcorner\rho) + \imath_1^{-1}\beta\lrcorner d'(\imath_1^{-1}\alpha\lrcorner\rho) - d'(\imath_1^{-1}\beta\lrcorner \imath_1^{-1}\alpha\lrcorner\rho) - \imath_1^{-1}\beta\lrcorner(\imath_1^{-1}\alpha\lrcorner d'\rho). $$

\end{lemm}

\begin{proof}
This lemma is a direct application of Corollary \ref{H1AB2}. We set
$$ \imath_1^{-1}\alpha= \alpha^{i,k}\overline{e}_i\otimes e_k^* , ~\imath_1^{-1}\beta= \beta^{j,l}\overline{e}_j\otimes e_l^*.$$
Then we set $ A = \alpha^{i,k} \overline{e}_i\cdot e_k^* $ and $ B = \beta^{j,l} \overline{e}_j\cdot e_k^* $. It is obvious that $ E $ in Corollary \ref{H1AB2} is taken as $ ^0T'' \otimes (^0\overline{T}'')^* $ in our case. For ease of notations, $ A= \alpha^i\cdot \overline{e}_i, ~ B= \beta^j\cdot \overline{e}_j $, here $ \alpha^i= \alpha^{i,k} e_k^*,~ \beta^j= \beta^{j,l} e_l^* $. More precisely, since\\
\begin{align}
\notag [A, B] &= [\alpha^i\cdot \overline{e}_i, \beta^j\cdot \overline{e}_j] \\
\notag &= [\alpha^i, \beta^j]\wedge \overline{e}_i\wedge \overline{e}_j +  [\overline{e}_i, \overline{e}_j]\wedge \alpha^i\wedge \beta^j - [\alpha^i, \overline{e}_j]\wedge \overline{e}_i\wedge \beta^j - [\overline{e}_i, \beta^j]\wedge \alpha^i\wedge \overline{e}_j \\
\notag &= - [\alpha^i, \overline{e}_j]\wedge \overline{e}_i\wedge \beta^j - [\overline{e}_i, \beta^j]\wedge \alpha^i\wedge \overline{e}_j \\
\notag &= \beta^j\wedge \overline{e}_j(\alpha^i)\wedge \overline{e}_i + \alpha^i\wedge \overline{e}_i(\beta^j)\wedge \overline{e}_j \\
 \notag &= \beta^{j,l}e_l^*\wedge \overline{e}_j(\alpha^{i,k})e_k^*\wedge \overline{e}_i + \alpha^{i,k}e_k^*\wedge \overline{e}_i(\beta^{j,l})e_l^*\wedge \overline{e}_j,
\end{align}
according to Lemma \ref{4.2}, we have
$$ [\imath_1^{-1}\alpha, \imath_1^{-1}\beta]= \frac{1}{2}\{\alpha^{j,k}\overline{e}_j(\beta^{i,l})\overline{e}_i\otimes e_k^*\wedge e_l^* - \beta^{j,l}\overline{e}_j(\alpha^{i,k})\overline{e}_i\otimes e_k^*\wedge e_l^* + \beta^{j,k}\overline{e}_j(\alpha^{i,l})\overline{e}_i\otimes e_k^*\wedge e_l^* - \alpha^{j,l}\overline{e}_j(\beta^{i,k})\overline{e}_i\otimes e_k^*\wedge e_l^* \}. $$
Then one has $ [A,B]\cdot\rho = [\imath_1^{-1}\alpha, \imath_1^{-1}\beta]\lrcorner \rho $. Moreover, one easily knows that
$$ A\cdot d'(B\cdot\rho) = \imath_1^{-1}\alpha \lrcorner d'(\imath_1^{-1}\beta \lrcorner \rho), $$
$$  B\cdot d'(A\cdot\rho) = \imath_1^{-1}\beta \lrcorner d'(\imath_1^{-1}\alpha \lrcorner \rho), $$
$$ d'(A\cdot B\cdot \rho) = d'(\imath_1^{-1}\beta \lrcorner \imath_1^{-1}\alpha \lrcorner \rho) ,$$
and $$ A\cdot B \cdot d'\rho = \imath_1^{-1}\beta \lrcorner (\imath_1^{-1}\alpha \lrcorner d'\rho). $$
Hence, by substituting the five equalities above into the formula \eqref{iH1AB2}, we complete our proof.
\end{proof}
So one has the obvious:
\begin{coro} [{\cite[Proposition 4.6]{[AM]}}]\label{CRTT}
If $ \alpha, \beta \in A^{n-2,1}(M) $ and satisfy $ d'\alpha = d'\beta = 0 $, then
$$ 2\imath_2 [\imath_1^{-1}\alpha, \imath_1^{-1}\beta] = - d'(\imath_1^{-1}\alpha \lrcorner \beta + \imath_1^{-1}\beta \lrcorner \alpha). $$
\end{coro}

\section{The deformation of CR-structures}

In this section, we first introduce the double complex $ (\mathbf{F}^{p,q}, d', d'')$ by Akahori-Miyajima \cite{[AM]}. Namely, one sets
$$
 \mathbf{F}^{p,q}= \{u : u\in \Gamma(M, (\mathbb{C}\zeta)^*\wedge\wedge^p(^0\overline{T}'')^*\wedge\wedge^q(^0T'')^*), d\theta\wedge(\zeta \lrcorner u)=0\}.
$$
Then by \eqref{eq4.1} and \eqref{eq4.2}, for $ u \in  \mathbf{F}^{p,q} $, we have $ d''u\in \mathbf{F}^{p,q+1} $ and $ d'u\in \mathbf{F}^{p+1,q}. $
Form now on we will study $(\mathbf{F}^{p,q}, d', d'')$. First, we have
$$ d'd'' + d''d'= 0\ \text{on $\mathbf{F}^{p,q}$}. $$
The next proposition describes  the relation of $T'$-valued tangential Cauchy-Riemann operator $\overline{\partial}_b$ with $d''$ on $M$. By the definition of $\overline{\partial}_b$(cf.\cite{[A1]}, \cite{[A2]}), for any $\phi\in \Gamma(M,\ ^0\overline{T}''\otimes \wedge ^q(^0T'')^*)$,\\
\begin{align}
 \notag \overline{\partial}_b\phi(e_{k_1}, e_{k_2},\cdot\cdot\cdot,e_{k_{q+1}})=&\sum\limits_{j=1}^{q+1}(-1)^{j-1}[e_{k_j}, \phi(e_{k_1},\cdot\cdot\cdot,  \widehat{e_{k_j}}, \cdot\cdot\cdot,e_{k_{q+1}})] \\
 \notag &+ \sum\limits_{j<l}(-1)^{j+1}\phi([e_{k_j}, e_{k_l}]_{^0T''}, e_{k_1},\cdot\cdot\cdot, \widehat{e_{k_j}}, \cdot\cdot\cdot, \widehat{e_{k_l}}, \cdot\cdot\cdot,e_{k_{q+1}}).
\end{align}

\begin{prop} [{\cite[Proposition 4.8]{[AM]}}]\label{5.1}
For $ \phi\in \Gamma(M,\ ^0\overline{T}''\otimes \wedge^q(^0T'')^*),$
$$ d''\circ \imath_q(\phi) = \imath_{q+1} \circ \overline{\partial}_b(\phi) + \sqrt{-1}d\theta\wedge (\zeta ~\lrcorner~ \imath_q\phi). $$
\end{prop}

\begin{prop} [{\cite[Proposition 4.9]{[AM]}}]\label{5.2}
If $ \alpha,\beta\in \mathbf{F}^{n-2,1}$, then $ \imath_1^{-1}\alpha \lrcorner \beta \in \mathbf{F}^{n-3,2}.$
\end{prop}

Let $ \mathbf{Z}^q $ be a subspace of $ \mathbf{F}^{n-2,q} $ given by
\begin{equation}\label{bf-z}
\mathbf{Z}^q= \{ \alpha\in \mathbf{F}^{n-2,q}:  d'\alpha= 0\}.
\end{equation}
 Obviously, $ d'' \mathbf{Z}^q \subset \mathbf{Z}^{q+1} $.

\begin{prop} [{\cite[Proposition 7.1]{[AM]}}]\label{5.3}
If $ \alpha\in \mathbf{Z}^1 $, then $ \imath_2\mathcal{P}(\imath_1^{-1}\alpha) \in \mathbf{Z}^2 $.
\end{prop}

\begin{proof}
By the definition of $ \mathcal{P}(\psi)$,
$$ \mathcal{P}(\imath_1^{-1}\alpha)= \overline{\partial}_b(\imath_1^{-1}\alpha) + \frac{1}{2}[\imath_1^{-1}\alpha, \imath_1^{-1}\alpha]. $$
So
\begin{align}
 \notag \imath_2 \mathcal{P}(\imath_1^{-1}\alpha) &= \imath_2\overline{\partial}_b(\imath_1^{-1}\alpha) + \imath_2\frac{1}{2}[\imath_1^{-1}\alpha, \imath_1^{-1}\alpha]\\
\notag &=d''\alpha - \frac{1}{2}d'(\imath_1^{-1}\alpha \lrcorner \alpha),
\end{align}
where the last equality follows from Proposition \ref{5.1} and Corollary \ref{CRTT}. Then we have
$$ d' \imath_2 \mathcal{P}(\imath_1^{-1}\alpha) = d'd''\alpha = -d''d'\alpha = 0.$$
\end{proof}

In Tian-Todorov's approach, $ \partial\overline{\partial}$-lemma for a compact K\"ahler manifold plays an essential role. We call the $ (\mathbf{F}^{p,q}, d', d'')$-version of $ \partial\overline{\partial}$-lemma the \emph{$d'd''$-lemma}. That is,\\

$d'd''$-$\mathbf{LEMMA}$. If $ \phi\in \mathbf{F}^{p,q} $ is $d''$-closed and $d'$-exact, or $ d'$-closed and $ d''$-exact, then it is $d'd''$-exact.\\

We use the notation
 \begin{equation} \label{assume}
 \mathbf{J}^{n-2,q} = (\ker d''\cap d'\mathbf{F}^{n-3,q})/ (d''\mathbf{F}^{n-2,q-1} \cap d'\mathbf{F}^{n-3,q}) ~~(2\leq q\leq n-1).
 \end{equation}
It is clear from the definition of $\mathbf{J}^{n-2,q}$ that if $d'd''$-lemma holds, then $\mathbf{J}^{n-2,q}=0~~(2\leq q \leq n-1).$
The next lemma, an analogy of \cite[Lemma 1.2.5]{[T1]}, is crucial for the proof  of the main theorem and distinguishes our proof from that of Akahori-Miyajima \cite{[AM]}.

\begin{lemm} \label{5.4}
Assume that $ \phi_k \in \mathbf{Z}^1$ and satisfy
$$ \overline{\partial}_b(\imath_1^{-1}\phi_k) = -\frac{1}{2} \sum\limits_{m+h=k} [\imath_1^{-1}\phi_m, \imath_1^{-1}\phi_h],$$
for any $k=2,\cdot\cdot\cdot,l$ and
$$ \overline{\partial}_b(\imath_1^{-1}\phi_1)=0.$$
Then one has
$$ \overline{\partial}_b(\sum\limits_{m+h=l+1} [\imath_1^{-1}\phi_m, \imath_1^{-1}\phi_h]) = 0. $$
\end{lemm}

\begin{proof} Compare \cite[Lemma 1.2.5]{[T1]} and also \cite[Lemma 4.2]{lry}.

Let $ \imath_1^{-1}\phi_i = \alpha_i = \alpha_i^{j,k}~\overline{e}_j\otimes e_k^* $.
By the definition of $ \overline{\partial}_b $ and Lemma \ref{4.1},
\begin{equation}\label{eq54}
\begin{aligned}
\overline{\partial}_b(\alpha_k)(e_{k_1}, e_{k_2}) &=[e_{k_1}, \alpha_k(e_{k_2})] - [e_{k_2}, \alpha_k(e_{k_1})] - \alpha_k([e_{k_1},e_{k_2}]_{^0T^{''}})  \\
 &=[e_{k_1}, \alpha_k^{i,k_2}\overline{e}_i] - [e_{k_2}, \alpha_k^{i,k_1}\overline{e}_i]\\
 &= e_{k_1}(\alpha_k^{i,k_2})\overline{e}_i - \alpha_k^{i,k_2} \sqrt{-1}\delta_{k_1,i}\zeta- e_{k_2}(\alpha_k^{i,k_1})\overline{e}_i + \alpha_k^{i,k_1} \sqrt{-1}\delta_{k_2,i}\zeta\\
 &= e_{k_1}(\alpha_k^{i,k_2})\overline{e}_i - e_{k_2}(\alpha_k^{i,k_1})\overline{e}_i - \alpha_k^{k_1,k_2} \sqrt{-1}\zeta + \alpha_k^{k_2,k_1} \sqrt{-1}\zeta.
\end{aligned}
\end{equation}
According to Lemma \ref{4.2}, one has
\begin{align*}
[\imath_1^{-1}\phi_m, \imath_1^{-1}\phi_h] &=  [\alpha_m^{s,t}~\overline{e}_s\otimes e_t^*, \alpha_h^{s,t}~\overline{e}_s\otimes e_t^*]\\
 &=\frac{1}{2}\{\alpha_m^{s,k_1}~\overline{e}_s(\alpha_h^{s,k_2}) - \alpha_h^{s,k_2}~\overline{e}_s(\alpha_m^{s,k_1}) + \alpha_h^{s,k_1}~\overline{e}_s(\alpha_m^{s,k_2}) - \alpha_m^{s,k_2}~\overline{e}_s(\alpha_h^{s,k_1})\} \overline{e}_i\otimes e_{k_1}^*\wedge e_{k_2}^*.
\end{align*}
So we have
\begin{equation}\label{eq55}
\begin{aligned}
 -\frac{1}{2} \sum\limits_{m+h=k}[\imath_1^{-1}\phi_m, \imath_1^{-1}\phi_h]=-\frac{1}{4}\{&\sum\limits_{m+h=k}\alpha_m^{s,k_1}~\overline{e}_s(\alpha_h^{s,k_2}) - \sum\limits_{m+h=k}\alpha_h^{s,k_2}~\overline{e}_s(\alpha_m^{s,k_1}) \\
 &\ + \sum\limits_{m+h=k}\alpha_h^{s,k_1}~\overline{e}_s(\alpha_m^{s,k_2}) - \sum\limits_{m+h=k}\alpha_m^{s,k_2}~\overline{e}_s(\alpha_h^{s,k_1})\} \overline{e}_i\otimes e_{k_1}^*\wedge e_{k_2}^*.
\end{aligned}
\end{equation}
By assumption, combining \eqref{eq54} with \eqref{eq55} yields that for $ k= 2,\cdot\cdot\cdot,l,$
\begin{equation}\label{eq56}
\begin{aligned}
e_{k_1}(\alpha_k^{i,k_2}) - e_{k_2}(\alpha_k^{i,k_1})=\ -\frac{1}{4}\{&\sum\limits_{m+h=k}\alpha_m^{s,k_1}~\overline{e}_s(\alpha_h^{s,k_2}) - \sum\limits_{m+h=k}\alpha_h^{s,k_2}~\overline{e}_s(\alpha_m^{s,k_1}) \\
 &\ + \sum\limits_{m+h=k}\alpha_h^{s,k_1}~\overline{e}_s(\alpha_m^{s,k_2}) - \sum\limits_{m+h=k}\alpha_m^{s,k_2}~\overline{e}_s(\alpha_h^{s,k_1})\},
\end{aligned}
\end{equation}
and
$$
\alpha_k^{k_1,k_2} = \alpha_k^{k_2,k_1}.
$$
Similarly,
\begin{equation}\label{eq58}
\begin{aligned}
 \sum\limits_{m+h=l+1}[\imath_1^{-1}\phi_m, \imath_1^{-1}\phi_h]=\ \frac{1}{2}\{&\sum\limits_{m+h=l+1}\alpha_m^{s,k_1}~\overline{e}_s(\alpha_h^{s,k_2}) - \sum\limits_{m+h=l+1}\alpha_h^{s,k_2}~\overline{e}_s(\alpha_m^{s,k_1}) \\
&\ + \sum\limits_{m+h=l+1}\alpha_h^{s,k_1}~\overline{e}_s(\alpha_m^{s,k_2}) - \sum\limits_{m+h=l+1}\alpha_m^{s,k_2}~\overline{e}_s(\alpha_h^{s,k_1})\} \overline{e}_i\otimes e_{k_1}^*\wedge e_{k_2}^*.
\end{aligned}
\end{equation}
For the first term on the RHS of \eqref{eq58}, one has
\begin{align*}
   \ &\overline{\partial}_b(\sum\limits_{m+h=l+1} \alpha_m^{s,k_1}\overline{e}_s(\alpha_h^{i,k_2})\overline{e}_i\otimes e_{k_1}^*\wedge e_{k_2}^*)(e_{l_1}, e_{l_2}, e_{l_3}) \\
 =\ &[e_{l_1}, \sum\limits_{m+h=l+1} \alpha_m^{s,l_2}\overline{e}_s(\alpha_h^{i,l_3})\overline{e}_i]
  - [e_{l_2}, \sum\limits_{m+h=l+1} \alpha_m^{s,l_1}\overline{e}_s(\alpha_h^{i,l_3})\overline{e}_i]
+ [e_{l_3}, \sum\limits_{m+h=l+1} \alpha_m^{s,l_1}\overline{e}_s(\alpha_h^{i,l_2})\overline{e}_i]\\
 =\ &\quad e_{l_1}(\sum\limits_{m+h=l+1} \alpha_m^{s,l_2}\overline{e}_s(\alpha_h^{i,l_3}))\overline{e}_i - \sum\limits_{m+h=l+1} \alpha_m^{s,l_2}\overline{e}_s(\alpha_h^{l_1,l_3})\sqrt{-1}\zeta \\
 \ &- e_{l_2}(\sum\limits_{m+h=l+1} \alpha_m^{s,l_1}\overline{e}_s(\alpha_h^{i,l_3}))\overline{e}_i
 + \sum\limits_{m+h=l+1} \alpha_m^{s,l_1}\overline{e}_s(\alpha_h^{l_2,l_3})\sqrt{-1}\zeta\\
  \ &+ e_{l_3}(\sum\limits_{m+h=l+1} \alpha_m^{s,l_1}\overline{e}_s(\alpha_h^{i,l_2}))\overline{e}_i
 - \sum\limits_{m+h=l+1} \alpha_m^{s,l_1}\overline{e}_s(\alpha_h^{l_3,l_2})\sqrt{-1}\zeta.
\end{align*}

%
%
%
%
So we have
\begin{align*}
&2\overline{\partial}_b(\sum\limits_{m+h=l+1} [\imath_1^{-1}\phi_m, \imath_1^{-1}\phi_h])(e_{l_1}, e_{l_2}, e_{l_3})\\
=\ &\quad  e_{l_1}(\sum\limits_{m+h=l+1} \alpha_m^{s,l_2}\overline{e}_s(\alpha_h^{i,l_3}))\overline{e}_i - \sum\limits_{m+h=l+1} \alpha_m^{s,l_2}\overline{e}_s(\alpha_h^{l_1,l_3})\sqrt{-1}\zeta \\
\ & -e_{l_2}(\sum\limits_{m+h=l+1} \alpha_m^{s,l_1}\overline{e}_s(\alpha_h^{i,l_3}))\overline{e}_i + \sum\limits_{m+h=l+1} \alpha_m^{s,l_1}\overline{e}_s(\alpha_h^{l_2,l_3})\sqrt{-1}\zeta \\
 \ &+e_{l_3}(\sum\limits_{m+h=l+1} \alpha_m^{s,l_1}\overline{e}_s(\alpha_h^{i,l_2}))\overline{e}_i - \sum\limits_{m+h=l+1} \alpha_m^{s,l_1}\overline{e}_s(\alpha_h^{l_3,l_2})\sqrt{-1}\zeta \\
 \ &-e_{l_1}(\sum\limits_{m+h=l+1} \alpha_h^{s,l_3}\overline{e}_s(\alpha_m^{i,l_2}))\overline{e}_i + \sum\limits_{m+h=l+1} \alpha_h^{s,l_3}\overline{e}_s(\alpha_m^{l_1,l_2})\sqrt{-1}\zeta \\
 \ &+e_{l_2}(\sum\limits_{m+h=l+1} \alpha_h^{s,l_3}\overline{e}_s(\alpha_m^{i,l_1}))\overline{e}_i - \sum\limits_{m+h=l+1} \alpha_h^{s,l_3}\overline{e}_s(\alpha_m^{l_2,l_1})\sqrt{-1}\zeta \\
 \ &-e_{l_3}(\sum\limits_{m+h=l+1} \alpha_h^{s,l_2}\overline{e}_s(\alpha_m^{i,l_1}))\overline{e}_i + \sum\limits_{m+h=l+1} \alpha_h^{s,l_2}\overline{e}_s(\alpha_m^{l_3,l_1})\sqrt{-1}\zeta \\
 \ &+e_{l_1}(\sum\limits_{m+h=l+1} \alpha_h^{s,l_2}\overline{e}_s(\alpha_m^{i,l_3}))\overline{e}_i - \sum\limits_{m+h=l+1} \alpha_h^{s,l_2}\overline{e}_s(\alpha_m^{l_1,l_3})\sqrt{-1}\zeta \\
\ &-e_{l_2}(\sum\limits_{m+h=l+1} \alpha_h^{s,l_1}\overline{e}_s(\alpha_m^{i,l_3}))\overline{e}_i + \sum\limits_{m+h=l+1} \alpha_h^{s,l_1}\overline{e}_s(\alpha_m^{l_2,l_3})\sqrt{-1}\zeta \\
\ &+e_{l_3}(\sum\limits_{m+h=l+1} \alpha_h^{s,l_1}\overline{e}_s(\alpha_m^{i,l_2}))\overline{e}_i - \sum\limits_{m+h=l+1} \alpha_h^{s,l_1}\overline{e}_s(\alpha_m^{l_3,l_2})\sqrt{-1}\zeta \\
\ & -e_{l_1}(\sum\limits_{m+h=l+1} \alpha_m^{s,l_3}\overline{e}_s(\alpha_h^{i,l_2}))\overline{e}_i + \sum\limits_{m+h=l+1} \alpha_m^{s,l_3}\overline{e}_s(\alpha_h^{l_1,l_2})\sqrt{-1}\zeta \\
\ &+e_{l_2}(\sum\limits_{m+h=l+1} \alpha_m^{s,l_3}\overline{e}_s(\alpha_h^{i,l_1}))\overline{e}_i - \sum\limits_{m+h=l+1} \alpha_m^{s,l_3}\overline{e}_s(\alpha_h^{l_2,l_1})\sqrt{-1}\zeta \\
\ &-e_{l_3}(\sum\limits_{m+h=l+1} \alpha_m^{s,l_2}\overline{e}_s(\alpha_h^{i,l_1}))\overline{e}_i + \sum\limits_{m+h=l+1} \alpha_m^{s,l_2}\overline{e}_s(\alpha_h^{l_3,l_1})\sqrt{-1}\zeta.
\end{align*}
Since $\alpha_k^{k_1,k_2} = \alpha_k^{k_2,k_1}$, the sum of terms with $ \zeta$ vanishes. One calculates
\begin{align}
 \notag RHS =\quad e_{l_1}(\sum\limits_{m+h=l+1} \alpha_m^{s,l_2}\overline{e}_s(\alpha_h^{i,l_3}))\overline{e}_i  -e_{l_2}(\sum\limits_{m+h=l+1} \alpha_m^{s,l_1}\overline{e}_s(\alpha_h^{i,l_3}))\overline{e}_i\\
\notag +e_{l_3}(\sum\limits_{m+h=l+1} \alpha_m^{s,l_1}\overline{e}_s(\alpha_h^{i,l_2}))\overline{e}_i
-e_{l_1}(\sum\limits_{m+h=l+1} \alpha_h^{s,l_3}\overline{e}_s(\alpha_m^{i,l_2}))\overline{e}_i\\
\notag +e_{l_2}(\sum\limits_{m+h=l+1} \alpha_h^{s,l_3}\overline{e}_s(\alpha_m^{i,l_1}))\overline{e}_i
-e_{l_3}(\sum\limits_{m+h=l+1} \alpha_h^{s,l_2}\overline{e}_s(\alpha_m^{i,l_1}))\overline{e}_i\\
\notag +e_{l_1}(\sum\limits_{m+h=l+1} \alpha_h^{s,l_2}\overline{e}_s(\alpha_m^{i,l_3}))\overline{e}_i
-e_{l_2}(\sum\limits_{m+h=l+1} \alpha_h^{s,l_1}\overline{e}_s(\alpha_m^{i,l_3}))\overline{e}_i\\
\notag -e_{l_3}(\sum\limits_{m+h=l+1} \alpha_h^{s,l_2}\overline{e}_s(\alpha_m^{i,l_1}))\overline{e}_i
-e_{l_1}(\sum\limits_{m+h=l+1} \alpha_m^{s,l_3}\overline{e}_s(\alpha_h^{i,l_2}))\overline{e}_i\\
\notag +e_{l_2}(\sum\limits_{m+h=l+1} \alpha_m^{s,l_3}\overline{e}_s(\alpha_h^{i,l_1}))\overline{e}_i
-e_{l_3}(\sum\limits_{m+h=l+1} \alpha_m^{s,l_2}\overline{e}_s(\alpha_h^{i,l_1}))\overline{e}_i. \\
\notag
\end{align}
From Lemma \ref{4.1}.(2), we have
\begin{equation}\label{}
\begin{aligned}
RHS = \sum\limits_{m+h=l+1}\{&\quad e_{l_1}(\alpha_m^{s,l_2}) \overline{e}_s(\alpha_h^{i,l_3})
+ \alpha_m^{s,l_2}\overline{e}_s e_{l_1}(\alpha_h^{i,l_3})
- \sqrt{-1}\alpha_m^{l_1,l_2}\zeta(\alpha_h^{i,l_3})\\
&-e_{l_2}(\alpha_m^{s,l_1}) \overline{e}_s(\alpha_h^{i,l_3})
- \alpha_m^{s,l_1}\overline{e}_s e_{l_2}(\alpha_h^{i,l_3})
+ \sqrt{-1}\alpha_m^{l_2,l_1}\zeta(\alpha_h^{i,l_3})\\
&+e_{l_3}(\alpha_m^{s,l_1}) \overline{e}_s(\alpha_h^{i,l_2})
+ \alpha_m^{s,l_1}\overline{e}_s e_{l_3}(\alpha_h^{i,l_2})
- \sqrt{-1}\alpha_m^{l_3,l_1}\zeta(\alpha_h^{i,l_2})\\
&-e_{l_1}(\alpha_h^{s,l_3}) \overline{e}_s(\alpha_m^{i,l_2})
- \alpha_h^{s,l_3}\overline{e}_s e_{l_1}(\alpha_h^{i,l_2})
+ \sqrt{-1}\alpha_h^{l_1,l_3}\zeta(\alpha_m^{i,l_2}) \\
&+e_{l_2}(\alpha_h^{s,l_3}) \overline{e}_s(\alpha_m^{i,l_1})
+ \alpha_h^{s,l_3}\overline{e}_s e_{l_2}(\alpha_m^{i,l_1})
- \sqrt{-1}\alpha_h^{l_2,l_3}\zeta(\alpha_m^{i,l_1})\\
& -e_{l_3}(\alpha_h^{s,l_2}) \overline{e}_s(\alpha_m^{i,l_1})
- \alpha_h^{s,l_2}\overline{e}_s e_{l_3}(\alpha_m^{i,l_1})
+ \sqrt{-1}\alpha_h^{l_3,l_2}\zeta(\alpha_m^{i,l_1})\\
& +e_{l_1}(\alpha_h^{s,l_2}) \overline{e}_s(\alpha_m^{i,l_3})
+ \alpha_h^{s,l_2}\overline{e}_s e_{l_1}(\alpha_m^{i,l_3})
- \sqrt{-1}\alpha_h^{l_1,l_2}\zeta(\alpha_m^{i,l_3})\\
&-e_{l_2}(\alpha_h^{s,l_1}) \overline{e}_s(\alpha_m^{i,l_3})
- \alpha_h^{s,l_1}\overline{e}_s e_{l_2}(\alpha_m^{i,l_3})
+\sqrt{-1}\alpha_h^{l_2,l_1}\zeta(\alpha_m^{i,l_3})\\
& +e_{l_3}(\alpha_h^{s,l_1}) \overline{e}_s(\alpha_m^{i,l_2})
+ \alpha_h^{s,l_1}\overline{e}_s e_{l_3}(\alpha_m^{i,l_2})
-\sqrt{-1}\alpha_h^{l_3,l_1}\zeta(\alpha_m^{i,l_2})\\
 &-e_{l_1}(\alpha_m^{s,l_3}) \overline{e}_s(\alpha_h^{i,l_2})
- \alpha_m^{s,l_3}\overline{e}_s e_{l_1}(\alpha_m^{i,l_2})
+\sqrt{-1}\alpha_m^{l_1,l_3}\zeta(\alpha_h^{i,l_2})\\
 &+e_{l_2}(\alpha_m^{s,l_3}) \overline{e}_s(\alpha_h^{i,l_1})
+ \alpha_m^{s,l_3}\overline{e}_s e_{l_2}(\alpha_h^{i,l_1})
-\sqrt{-1}\alpha_m^{l_2,l_3}\zeta(\alpha_h^{i,l_1})\\
 &-e_{l_3}(\alpha_m^{s,l_2}) \overline{e}_s(\alpha_h^{i,l_1})
- \alpha_m^{s,l_2}\overline{e}_s e_{l_3}(\alpha_h^{i,l_1})
+\sqrt{-1}\alpha_m^{l_3,l_2}\zeta(\alpha_h^{i,l_1})\}\overline{e}_i,
\end{aligned}
\end{equation}
where the sum of terms with $ \zeta $ vanishes again. From assumption and \eqref{eq56} it follows that
\begin{align*}
   \ &e_{l_1}(\alpha_m^{s,l_2}) \overline{e}_s(\alpha_h^{i,l_3}) - e_{l_2}(\alpha_m^{s,l_1}) \overline{e}_s(\alpha_h^{i,l_3}) \\
 =\ &-\frac{1}{4} \overline{e}_s(\alpha_h^{i,l_3}) \{\sum\limits_{p+q=m}\alpha_p^{j,l_1}\overline{e}_j(\alpha_q^{s,l_2}) -\sum\limits_{p+q=m}\alpha_p^{j,l_2}\overline{e}_j(\alpha_q^{s,l_1}) +\sum\limits_{p+q=m}\alpha_q^{j,l_1}\overline{e}_j(\alpha_p^{s,l_2})-\sum\limits_{p+q=m}\alpha_q^{j,l_2}\overline{e}_j(\alpha_p^{s,l_1}) \},
\end{align*}

%
\begin{align*}
   \ &e_{l_3}(\alpha_m^{s,l_1}) \overline{e}_s(\alpha_h^{i,l_2}) - e_{l_1}(\alpha_m^{s,l_3}) \overline{e}_s(\alpha_h^{i,l_2}) \\
 =\ &-\frac{1}{4} \overline{e}_s(\alpha_h^{i,l_2}) \{\sum\limits_{p+q=m}\alpha_p^{j,l_3}\overline{e}_j(\alpha_q^{s,l_1}) -\sum\limits_{p+q=m}\alpha_p^{j,l_1}\overline{e}_j(\alpha_q^{s,l_3}) +\sum\limits_{p+q=m}\alpha_q^{j,l_3}\overline{e}_j(\alpha_p^{s,l_1})-\sum\limits_{p+q=m}\alpha_q^{j,l_1}\overline{e}_j(\alpha_p^{s,l_3}) \},
\end{align*}

\begin{align*}
   \ &e_{l_2}(\alpha_m^{s,l_3}) \overline{e}_s(\alpha_h^{i,l_1}) - e_{l_3}(\alpha_m^{s,l_2}) \overline{e}_s(\alpha_h^{i,l_1}) \\
 =\ &-\frac{1}{4} \overline{e}_s(\alpha_h^{i,l_1}) \{\sum\limits_{p+q=m}\alpha_p^{j,l_2}\overline{e}_j(\alpha_q^{s,l_3}) -\sum\limits_{p+q=m}\alpha_p^{j,l_3}\overline{e}_j(\alpha_q^{s,l_2}) +\sum\limits_{p+q=m}\alpha_q^{j,l_2}\overline{e}_j(\alpha_p^{s,l_3})-\sum\limits_{p+q=m}\alpha_q^{j,l_3}\overline{e}_j(\alpha_p^{s,l_2}) \},
\end{align*}
\begin{align*}
   \ & \alpha_m^{s,l_1} \overline{e}_s\{ e_{l_3}(\alpha_h^{i,l_2}) - e_{l_2}(\alpha_h^{i,l_3})\}\\
 =\ &-\frac{1}{4} \alpha_m^{s,l_1} \overline{e}_s \{ \sum\limits_{p+q=h}\alpha_p^{j,l_3}\overline{e}_j(\alpha_q^{i,l_2}) -\sum\limits_{p+q=h}\alpha_q^{j,l_2}\overline{e}_j(\alpha_p^{i,l_3}) +\sum\limits_{p+q=h}\alpha_q^{j,l_3}\overline{e}_j(\alpha_p^{i,l_2})-\sum\limits_{p+q=h}\alpha_p^{j,l_2}\overline{e}_j(\alpha_q^{s,l_3}) \}\\
 =\ &-\frac{1}{4} \alpha_m^{s,l_1} \sum\limits_{p+q=h} \{ \overline{e}_s(\alpha_p^{j,l_3})\overline{e}_j(\alpha_q^{i,l_2})
- \overline{e}_s(\alpha_q^{j,l_2})\overline{e}_j(\alpha_p^{i,l_3})
+ \overline{e}_s(\alpha_q^{j,l_3})\overline{e}_j(\alpha_p^{i,l_2}) -\overline{e}_s(\alpha_p^{j,l_2}))\overline{e}_j(\alpha_q^{i,l_3}) \}\\
  &-\frac{1}{4} \alpha_m^{s,l_1} \sum\limits_{p+q=h} \{ \alpha_p^{j,l_3} \overline{e}_s \overline{e}_j(\alpha_q^{i,l_2})
- \alpha_q^{j,l_2} \overline{e}_s \overline{e}_j(\alpha_p^{i,l_3})
+ \alpha_q^{j,l_3} \overline{e}_s \overline{e}_j(\alpha_p^{i,l_2})
- \alpha_p^{j,l_2} \overline{e}_s \overline{e}_j(\alpha_q^{i,l_3}) \},
\end{align*}
\begin{align*}
   \ &\alpha_m^{s,l_3} \overline{e}_s\{ e_{l_2}(\alpha_h^{i,l_1}) - e_{l_1}(\alpha_h^{i,l_2})\} \\
 =\ &-\frac{1}{4} \alpha_m^{s,l_3} \overline{e}_s \{ \sum\limits_{p+q=h}\alpha_p^{j,l_2}\overline{e}_j(\alpha_q^{i,l_1}) -\sum\limits_{p+q=h}\alpha_q^{j,l_1}\overline{e}_j(\alpha_p^{i,l_2}) +\sum\limits_{p+q=h}\alpha_q^{j,l_2}\overline{e}_j(\alpha_p^{i,l_1})-\sum\limits_{p+q=h}\alpha_p^{j,l_1}\overline{e}_j(\alpha_q^{s,l_2}) \}\\
 =\ &-\frac{1}{4} \alpha_m^{s,l_3} \sum\limits_{p+q=h} \{ \overline{e}_s(\alpha_p^{j,l_2})\overline{e}_j(\alpha_q^{i,l_1})
- \overline{e}_s(\alpha_q^{j,l_1})\overline{e}_j(\alpha_p^{i,l_2})
+ \overline{e}_s(\alpha_q^{j,l_2})\overline{e}_j(\alpha_p^{i,l_1}) -\overline{e}_s(\alpha_p^{j,l_1}))\overline{e}_j(\alpha_q^{i,l_2}) \}\\
  &-\frac{1}{4} \alpha_m^{s,l_3} \sum\limits_{p+q=h} \{ \alpha_p^{j,l_2} \overline{e}_s \overline{e}_j(\alpha_q^{i,l_1})
- \alpha_q^{j,l_1} \overline{e}_s \overline{e}_j(\alpha_p^{i,l_2})
+ \alpha_q^{j,l_2} \overline{e}_s \overline{e}_j(\alpha_p^{i,l_1})
- \alpha_p^{j,l_1} \overline{e}_s \overline{e}_j(\alpha_q^{i,l_2}) \},
\end{align*}
\begin{align*}
   \ & \alpha_m^{s,l_2} \overline{e}_s\{ e_{l_1}(\alpha_h^{i,l_3}) - e_{l_3}(\alpha_h^{i,l_1})\}\\
 =\ &-\frac{1}{4} \alpha_m^{s,l_2} \overline{e}_s \{ \sum\limits_{p+q=h}\alpha_p^{j,l_1}\overline{e}_j(\alpha_q^{i,l_3}) -\sum\limits_{p+q=h}\alpha_q^{j,l_3}\overline{e}_j(\alpha_p^{i,l_1}) +\sum\limits_{p+q=h}\alpha_q^{j,l_1}\overline{e}_j(\alpha_p^{i,l_3})-\sum\limits_{p+q=h}\alpha_p^{j,l_3}\overline{e}_j(\alpha_q^{s,l_1}) \}\\
 =\ &-\frac{1}{4} \alpha_m^{s,l_2} \sum\limits_{p+q=h} \{ \overline{e}_s(\alpha_p^{j,l_1})\overline{e}_j(\alpha_q^{i,l_3})
- \overline{e}_s(\alpha_q^{j,l_3})\overline{e}_j(\alpha_p^{i,l_1})
+ \overline{e}_s(\alpha_q^{j,l_1})\overline{e}_j(\alpha_p^{i,l_3}) -\overline{e}_s(\alpha_p^{j,l_3}))\overline{e}_j(\alpha_q^{i,l_1}) \}\\
  &-\frac{1}{4} \alpha_m^{s,l_2} \sum\limits_{p+q=h} \{ \alpha_p^{j,l_1} \overline{e}_s \overline{e}_j(\alpha_q^{i,l_3})
- \alpha_q^{j,l_3} \overline{e}_s \overline{e}_j(\alpha_p^{i,l_1})
+ \alpha_q^{j,l_1} \overline{e}_s \overline{e}_j(\alpha_p^{i,l_3})
- \alpha_p^{j,l_3} \overline{e}_s \overline{e}_j(\alpha_q^{i,l_1}) \}.
\end{align*}

The comparison of the six equalities above and a brute-force routine computation gives
$$ \overline{\partial}_b(\sum\limits_{m+h=l+1} [\imath_1^{-1}\phi_m, \imath_1^{-1}\phi_h]) = 0. $$
\end{proof}

Now we can prove the main theorem of this paper.
\begin{theo} [{\cite[MAIN THEOREM]{[AM]}}]
Let $ (M,\ ^0T'') $ be a normal strongly pseudoconvex CR-manifold with $ \dim_{\mathbb{R}}M  = 2n-1 \geq 7 $. And we assume that its canonical line bundle $ K_M = \wedge^n(T')^* $ is trivial in CR-sense. Then the obstructions in $ \imath_1^{-1}(\mathbf{Z}^1) $ appear in $ \mathbf{J}^{n-2,2} $ which defined as \eqref{assume}. That is, if $ \mathbf{J}^{n-2,2}=0 $, then any deformation of CR structures in $ \imath_1^{-1}(\mathbf{Z}^1) $ is unobstructed.
\end{theo}

\begin{proof}
We will construct a $ \mathbf{Z}^1$-valued polynomial $$ \phi^l(t)=\sum\limits_{k=1}^l\phi_k(t) $$ in $ t $, satisfying $ \mathcal{P}(\imath_1^{-1}\phi^l(t))=0$ for all $l$. For convenience, sometimes we omit the $t$ in $\phi_k(t)$ in the sequel. According to Proposition \ref{5.3}, it is equivalent to solve the system of equations
$$ d''\phi_k - \frac{1}{2}\sum\limits_{m+h=k} d'(\imath_1^{-1}\phi_m \lrcorner \phi_h) = 0. $$

For $ k=1$, we assume $\imath_1^{-1}\phi_1\in H_{\overline{\partial}_b}^1(M, T')$, that is, $\overline{\partial}_b(\imath_1
^{-1}\phi_1)=0 $.

For $ k=2$, we want to solve $\phi_2(t)$ in equation
$$ d''\phi_2 - \frac{1}{2} d'(\imath_1^{-1}\phi_1 \lrcorner \phi_1) = 0. $$
Since $ d''\phi_1=d''\imath_1\imath_1^{-1}\phi_1= \imath_2\overline{\partial}_b(\imath_1^{-1}\phi_1) = 0$ and similarly $ d''(\imath_1^{-1}\phi_1) = 0, $ we have
$$ d''d'(\imath_1^{-1}\phi_1 \lrcorner \phi_1) = -d'd''(\imath_1^{-1}\phi_1 \lrcorner \phi_1) = -d'(d''(\imath_1^{-1}\phi_1) \lrcorner \phi_1)- d'(\imath_1^{-1}\phi_1 \lrcorner d''\phi_1) = 0. $$
Then we can find $ \phi_2(t)$ by the assumption $ \mathbf{J}^{n-2,2}=0 $.

By induction, we may assume that the equation is solved for $ k \leq l $ and we have constructed $ \phi_k(t)$, $ k \leq l $. For $ k= l+1 $, according to Corollary \ref{CRTT},
$$ \sum\limits_{m+h=l+1} d''d'(\imath_1^{-1}\phi_m \lrcorner \phi_h)= \sum\limits_{m+h=l+1} -d''[\imath_1^{-1}\phi_m, \imath_1^{-1}\phi_h] \lrcorner \omega, $$
where $ \omega\in\Gamma(M,\wedge^n(T')^*)$ is a nowhere vanishing section satisfying $ d''\omega=0 $ as Definition \ref{imath-q}. Then by  Proposition \ref{5.1}, we have
$$\sum\limits_{m+h=l+1} -d''[\imath_1^{-1}\phi_m, \imath_1^{-1}\phi_h] \lrcorner \omega= \overline{\partial}_b(\sum\limits_{m+h=l+1} [\imath_1^{-1}\phi_m, \imath_1^{-1}\phi_h])\lrcorner \omega =0, $$
where the last equality follows from Lemma \ref{5.4}. Then by assumption we can solve $ \phi_{l+1}(t)$.

By a canonical choice of $\phi_\bullet$ and the same argument as in \cite{[A3]}, we can prove the convergence of $ \phi(t) = \sum\limits_{k=1}^{+\infty}\phi_k(t) $ with respect to the Folland-Stein norm (cf.\cite{[A2]}).
\end{proof}

\begin{coro} [{\cite[Corollary 9.1]{[AM]}}]
Suppose that $\dim_{\mathbb{R}}M \geq 7 $. If $K_M$ is trivial in CR-sense and if $d'd''$-lemma holds in $(\mathbf{F}^{p,q}, d', d'')$, then all Kuranishi families of strongly pseudoconvex CR-structures are unobstructed.
\end{coro}

\section*{Acknowledgement}
This work was mainly completed during the authors' visit to Institute of Mathematics, Academia Sinica in the summer of 2018.
They would like to express their gratitude to the institute for their hospitality and the wonderful work environment during their visit, especially Professors Jih-Hsin Cheng and Chin-Yu Hsiao for many discussions on CR geometry. Last, we would like to thank Professor K. Miyajima for a useful comment on our paper.


\end{document}